\documentclass{amsart}
\usepackage{amssymb}
\usepackage{amsmath}
\usepackage{amsfonts}
\newtheorem{thm}{Theorem}[section]

\newtheorem{cor}[thm]{Corollary}

\newtheorem{prop}[thm]{Proposition}

\theoremstyle{definition}

\theoremstyle{definition}
\newtheorem{rem}[thm]{Remark}
\theoremstyle{definition}

\begin{document}
\title{On the regularity of the $2+1$ dimensional Skyrme model}

\author{Dan-Andrei Geba and  Daniel da Silva}
\address{Department of Mathematics, University of Rochester, Rochester, NY 14627}
\address{Department of Mathematics, University of Rochester, Rochester, NY 14627}
\date{\today}

\begin{abstract}
One of the most interesting open problems concerning the Skyrme model of nuclear physics is the regularity of its solutions (\cite{MR2376667}). In this article, we study $2+1$ dimensional equivariant Skyrme maps, for which we prove, using the method of multipliers, that the energy does not concentrate. This is one of the crucial steps towards a global regularity theory. 
\end{abstract}

\subjclass[2000]{35L70, 81T13}
\keywords{Skyrme model, global existence, non-concentration of energy.}
\maketitle

\section{Introduction}
The nonlinear $\sigma$ model of quantum field theory  is given by the Lagrangian
\begin{equation}
L_{\sigma}\,=\,-\frac 12 S^{\mu}_{\mu}, \quad \text{where} \quad S_{\mu\nu}\,=\,g_{AB}\,\partial_\mu U^A \,\partial_\nu U^B
\label{sigma}
\end{equation}
is the pulled-back metric corresponding to $U: M \to N$, a map from a spacetime $(M, m_{\mu\nu})$ to a complete Riemannian manifold $(N, g_{AB})$. The critical points for this Lagrangian, also called harmonic maps or wave maps (depending on whether $m$ is a Riemannian or Lorentzian metric), have been a major research theme in the general area of partial differential equations for many years (see the excellent surveys: \cite{MR1913803} for harmonic maps and \cite{MR2043751}, \cite{MR2175916}, and \cite{MR2488946} for wave maps).

For the purposes of this article, we specialize our discussion only to wave maps, with $(M,m)=(\mathbb{R}^{n+1}, \text{diag}(-1,1,\ldots,1))$, which have a priori a conserved energy, 
\begin{equation}
E(U)\,=\,\frac 12 \,\int_{\mathbb{R}^n}\,|U_t|^2_g\,+\,|\nabla_x U|^2_g\ dx.
\label{en}
\end{equation}
The case when $n=3$ and $N=\mathbb{S}^3\subset\mathbb{R}^4$ is referred by the physicists as the \emph{classical} nonlinear $\sigma$ model (\cite{MR0118418}, \cite{MR0140316}, \cite{MR0159685}), describing the interactions between nucleons and pions. For finite energy wave maps, one can naturally associate to them a  winding number
\begin{equation}
Q\,=\, c\, \int_{\mathbb{R}^3}\, \epsilon_{ijk}\,\partial_a U^i\, \partial_b U^j\, \partial_c U^k \,\epsilon^{abc}\,dx,
\label{Q}
\end{equation}
where $\epsilon$ is the Levi-Civita symbol and $c$ is a normalizing constant such that $Q$ takes integer values. As the energy, this is also a conserved quantity (or \emph{topological charge} in physical terminology), and it was Tony Skyrme's revolutionary idea which saw it as nothing but the baryon number of the nucleus, making it the first example of a \emph{topological soliton}.

From the mathematical point of view, one of the fundamental questions related to wave maps is their global regularity, of particular interest being the classical model and the one corresponding to  $n=2$ and $N=\mathbb{S}^2\subset\mathbb{R}^3$. For the former, Shatah (\cite{MR933231}) showed that singularities form in finite time, even for smooth initial data, which is in accord with the physical intuition that, due to the attractive nature of the forces between pions, finite energy initial configurations of topological charge equal to one would shrink to a point, leading to a singularity formation. The $2+1$ dimensional case proved much more challenging and the same result was obtained only recently (e.g., \cite{MR2372807}, \cite{MR2680419}, and \cite{RR}). 

Another major contribution of Skyrme is his unifying model for mesons and baryons (\cite{MR0128862}, \cite{MR0153323}, \cite{MR0138394}), which is a generalization of the classical model, with
\begin{equation}
L_{S}\,=\,L_\sigma \,+\, \frac{\alpha^2}{4} (S^{\mu\nu} S_{\mu\nu}\,-\,S^{\mu}_{\mu}S^{\nu}_{\nu}),\label{S}
\end{equation}
where $\alpha$ is a constant having the dimension of length. This theory is meant to fix the shortcomings of the classical model pointed out by Shatah; it has stable topological solitons, also called \emph{skyrmions}. Therefore, the mathematical expectations with respect to the Euler-Lagrange equations associated to \eqref{S}, are that finite energy initial data evolve into global smooth solutions, which is one of the most interesting open problems concerning the Skyrme model (\cite{MR2376667} and references therein).  Numerical evidence supporting this claim can be found in \cite{BCR}.  

In this article, we take the first step towards proving  global regularity for a Skyrme theory in $2+1$ dimensions (with $N=\mathbb{S}^2\subset\mathbb{R}^3$), which, in order to be Lorentz invariant and to have stable soliton-like solutions, needs an additional potential term in \eqref{S}:
\begin{equation}
\tilde{L}_S\,=\,L_S \,-\,V(U).
\label{S2}
\end{equation}
The choices for $V$ vary in literature, e.g.,  $V(U)= \lambda^2 (1-n\cdot U)$ \cite{MR1449544} or  $V(U)= \lambda^2 (1-n\cdot U)^2$ \cite{MR2070206}, where $\lambda$ is a coupling constant having the inverse dimension of length and $n=(0,0,1)$ is the north pole of $\mathbb{S}^2$. However, the physical insight says that in what concerns the short-distance behavior (e.g., the development of singularities), the presence/absence of this new term is irrelevant, which is shown precisely later in Remark \ref{pot}. This is why we decide to ignore the potential and work with the original Skyrme Lagrangian \eqref{S}.
 
We study equivariant critical maps of co-rotation index 1 associated to $L_S$, i.e.,
\[
U(t,r,\omega)= (u(t,r),\omega), \qquad u(t,0)=0, \qquad u(t,\infty)= \pi, 
\]
where $u$ is the longitudinal angle and $(r,\omega)$ are the polar coordinates on $\mathbb{R}^2$. The Euler-Lagrange equations yield the following quasilinear wave equation for $u$:
\begin{equation}
\label{meq}
\aligned
\left(1+\frac{\alpha^2\sin^2u}{r^2}\right)&(u_{tt}-u_{rr}) - \left(1-\frac{\alpha^2\sin^2u}{r^2}\right)\frac{u_r}{r}\\
&+ \frac{\sin 2u}{2r^2}\left[\alpha^2\left(u_t^2-u_r^2\right) + 1\right] = 0.
\endaligned
\end{equation}
A formal calculus shows that  the conserved energy norm is given by
\begin{equation}
\mathcal{E}[u](t)=\int_0^\infty \left[\left(1 + \frac{\alpha^2\sin ^2 u}{r^2}\right)\frac{u_t^2+u_r^2}{2}+
\frac{\sin ^2 u}{2r^2}
\right]\,r dr. \label{tote}
\end{equation}

We stress here that what we are after is a large data theory, for which the usual strategy is to combine a small data global well-posedness result with an argument that rules out energy concentration. In this article we achieve the latter, which is the more involved and non-standard part of the proof, while the well-posedness part is addressed, more generally, in an upcoming paper \cite{GNR}.

\section{Preliminaries}
Our main equation \eqref{meq} is only at first glance quasilinear, as one can divide by the coefficient of the leading terms, which is non-degenerate,  and obtain a $1+1$ dimensional semilinear wave equation, for which the short-time existence of smooth solutions is, by now, standard.  Moreover, from the radial symmetry, it is clear that if  \eqref{meq} develops a singularity, this has to appear at $r=0$. 

Therefore, as the equation is reversible in time and it is invariant under time translations, we can assume, without any loss of generality, that our solution starts at time $t =T_0>0$ and the origin $(t,r)=(0, 0)$ is the first possible blow-up point. Using a finite speed of propagation argument, we can thus prescribe initial data at $t=T_0$ and $r\leq T_0$, and solve \eqref{meq} backward in time in 
\[
\Omega\,=\,\{(t,r)|\ 0 < t\leq T_0,\, r \leq t\}.
\]
The local theory allows us also to assume that $u$ is smooth in $\Omega$. Under these conditions, our main result can be formulated as follows:
\begin{thm}
For $u$ a smooth solution of \eqref{meq} in $\Omega$, the energy doesn't concentrate at the origin, i.e.,
\begin{equation}
\lim_{T\to 0+}\,\int_0^T \left[\left(1 + \frac{\alpha^2\sin ^2 u}{r^2}\right)\frac{u_t^2+u_r^2}{2}+
\frac{\sin ^2 u}{2r^2}
\right]\,r dr\,=\,0.
\label{nce}
\end{equation}
\label{eth}
\end{thm}

Our energy analysis is mainly done on forward truncated cones, their mantels, and their spacelike sections, denoted by
\[
\aligned
K(t_1,t_2)\,&:=\,\{(t,r)|\ t_1\leq t\leq t_2,\, 0\leq r \leq t\},\\
C(t_1,t_2)\,&:=\,\{(s,s)|\ t_1\leq s\leq t_2\},\\
\Sigma(t_1)\,&:=\,\{(t_1,r)|\ 0\leq r \leq t_1\},
\endaligned
\]
where $0 <  t_1 \leq t_2 \leq T_0$. For narrower sections, we use 
\[
\Sigma_\lambda(t_1)\,:=\,\{(t_1,r)|\ \lambda t_1\leq r \leq t_1\},
\]
with $0< \lambda < 1$.

Next, using the notation
\[
w:= 1+ \frac{\alpha^2\sin^2 u}{r^2}, \qquad
e:= w\frac{u_t^2+u_r^2}{2} + \frac{\sin^2 u}{2r^2}, \qquad
m:=wu_tu_r,
\]
where $e$ and $m$ are also called the \emph{energy density}, respectively the \emph{momentum density}, we can record our first

\begin{prop}
Classical solutions of \eqref{meq} satisfy the following two differential identities:

\begin{equation}
\aligned \partial_t\left(r\left[ a\,e + b\,m+ c\, w\, h(u) u_t \right]\right) &-\partial_r\left(r \left[a\,m+b\left(e-\frac{\sin^2u}{r^2}\right)+ c\,w\,h(u) u_r\right]\right) \\
= \, r\,\Bigg\{ (A+B) \frac{u_{t}^2}{2} + 
(A-B&) \frac{u_{r}^2}{2} + \, (b_t-a_r)m +\,\left(a_t + b_r- \frac{b}{r}\right)\,\frac{\sin^2u}{2r^2} \\
-c\,\frac{h(u)\sin2u}{2r^2} &+ \,w\,h(u)(c_t u_t-c_r u_r)\Bigg\},
\endaligned \label{abc}
\end{equation}
where 
\[
\aligned
a=a(t,r), \qquad b=b(t,r), \qquad &c=c(t,r), \qquad h=h(u),\\
A\,:=\,w(a_t-b_r), \quad B\,:=\,-\left(1- \frac{\alpha^2\sin^2 u}{r^2}\right)&\frac{b}{r} +
c\left(2wh'(u)+\frac{\alpha^2h(u)\sin 2u}{r^2}\right),
\endaligned
\]

and

\begin{equation}
\aligned &\partial_t\left(r\,u\,u_t \right) -\partial_r \left(r \,u \,u_r\right) \\
= \, r\,\Bigg\{ &\left[1-\frac{\frac{\alpha^2 u\sin 2u}{2r^2}}{w}\right]\,(u_{t}^2 -u_r^2) +\left[\frac{1-\frac{\alpha^2 \sin^ 2u}{r^2}}{w}-1\right]\,\frac{uu_r}{r} -\frac{\frac{ u\sin 2u}{2r^2}}{w}\Bigg\}.
\endaligned \label{ru}
\end{equation}

\label{mty}
\end{prop}

\begin{proof}
The first identity is obtained when we multiply our equation by $r(a u_t + b u_r + c h(u))$ and rearrange the terms conveniently,  while the second one follows by rewriting $u_{tt}-u_{rr}$ using directly \eqref{meq}.
\end{proof}

For $(a,b,c)=(1,0,0)$,  \eqref{abc} is the energy differential identity
\begin{equation}
\partial_t(re)\,-\,\partial_r(rm)  \,=\, 0, 
\label{en}
\end{equation}
which, integrated on $K(t_1,t_2)\subset \Omega$, yields
\[
E(t_2)\,-\,E(t_1)\,=\,F(t_1,t_2).
\]
Here
\[
E(T)\,=\,\int_{\Sigma(T)}\,e \qquad \text{and} \qquad F(S,T)\,=\,\frac{1}{\sqrt 2}\int_{C(S,T)}\,e+m\]
are  the \emph{energy} of time slice $t=T$, respectively the \emph{flux} between the time slices $t=S$ and $t=T$ for the function $u$\footnote{This is the motivation for the terminology associated with $e$ and $m$.}. We thus obtain:

\begin{prop}
The energy is monotone and the flux decays to 0, i.e.,
\begin{equation}\label{mone}
E(t_1)\leq E(t_2)\quad \text{for}\quad 0 <  t_1 \leq t_2 \leq T_0
\end{equation}
\begin{equation}
\lim_{T\to 0+}F(0,T)\,=\,0 \label{flux}
\end{equation}
\label{ef}
\end{prop}

\begin{rem}
It follows immediately from \eqref{mone} that  
\begin{equation}
0 \leq \lim_{T\to 0+}E(T)\,=\,l < \infty, \label{ne}
\end{equation}
and so, in order to prove the non-concentration of energy, it suffices to demonstrate $l=0$. 
\end{rem}

We have now all the ingredients to show that:
\begin{prop}
A smooth solution $u$ for \eqref{meq} in $\Omega$ is continuous at $(0,0)$ and 
\begin{equation}
|u(t,r)|\,\leq C(E(T_0)) \,r^{1/2}, \quad (\forall)(t,r) \in \Omega, \, t\ll 1.
\label{dcy}
\end{equation}
\label{cont}
\end{prop}
\begin{proof}
For the functional
\[
I(z)=\int_0^z |\sin w|\,dw,
\]
one has
\[
I(0)=0, \qquad |I(z)|>0\,(z\neq 0),\qquad \lim_{|z|\to\infty}
|I(z)|=\infty.
\]
Using the co-rotational hypothesis (i.e., $u(t,0)= 0$), we can write
\[
I(u(t,r))=\int_0^r |\sin u(t,s)|\, u_r(t,s)\,ds,
\]
which implies, based on \eqref{mone},
\begin{equation}
\aligned
|I(u(t,r))|\,&\lesssim \,\left(\int_0^r s\,ds \right)^\frac 12 \
\left(\int_{\Sigma(t)} \,\frac{\sin^2u(t,s)}{s^2}u^2_r(t,s)\right)^\frac 12\\
&\lesssim\, r\,E(t)^{1/2}\,\lesssim\, r\,E(T_0)^{1/2}.
\endaligned
\label{dcy2}\end{equation}

This obviously shows that $u$ is continuous at the origin, i.e.,
\[
\lim_{(t,r)\in\Omega \to (0,0)}\ u(t,r)\,=\, 0.\]
For $t$ sufficiently small, we can then argue that 
\[
u^2(t,r)\, \lesssim\, \int_0^r |\sin u(t,s)|\cdot |u_r(t,s)|\,ds\,\lesssim\, r\,E(T_0)^{1/2}, 
\] 
which proves \eqref{dcy}.
\end{proof}

As an easy consequence, we obtain non-concentration for one of the terms in $e$:
\begin{cor}
For a smooth solution $u$ of \eqref{meq} in $\Omega$, the following holds
\begin{equation}
\lim_{T\to 0+}\,\int_{\Sigma(T)}\,\frac{\sin ^2 u}{2r^2}\,=\,0.
\label{sinu}
\end{equation}
\end{cor}

\begin{rem}
The presence of the potential term in \eqref{S} would add to $e$ only bounded expressions, which obviously do not concentrate. For example, if $V(U)= \lambda^2 (1-n\cdot U)$, then
\[
\lim_{T\to 0+}\,\int_{\Sigma(T)}\,\lambda^2(1-\cos u)\,=\,0.\]
\label{pot}
\end{rem}

\begin{rem}
From \eqref{dcy2}, we see that once we have a sequence $(t_n)_n\to 0$ with
\begin{equation}
\lim_{n\to \infty}\,\int_{\Sigma(t_n)}\,  \frac{\sin^2u}{r^2}u^2_r\,=\,0, 
\label{sinur}
\end{equation}
it follows that
\begin{equation}
|u(t_n,r)|\,\leq\, C_n\, r^{1/2}, \quad (\forall)(t_n,r) \in \Sigma(t_n),
\label{dcy3}\end{equation}
with $C_n\to 0$ as $n\to \infty$.
\end{rem}

\begin{rem} For general, finite energy smooth solutions of  \eqref{meq}, we deduce, using the same functional as above, that 
\begin{equation}
|I(u(t,r))|\lesssim \left(\int_0^r \frac{\sin^2u(t,s)}{s}\,ds \right)^\frac 12\cdot
\left(\int_0^r u^2_r(t,s) s\,ds\right)^\frac 12\,
\lesssim \,\mathcal{E}[u](t).\label{li}\end{equation}
This implies that $u$ is a priori bounded, as the energy norm \eqref{tote} is conserved in time.
\end{rem}

\section{Main argument}
We follow closely the blueprint for proving non-concentration of energy, as it was used for equivariant wave maps  in \cite{MR1168115}. First, we show that the energy doesn't concentrate near the light cone, which implies, jointly with \eqref{sinu}, that the energy corresponding to the $u_t^2$ and $ \frac{\sin^2u}{r^2}u^2_r$ terms decays to $0$. This result is then used to obtain non-concentration for the $u_r^2$ density. Finally, we rely on all these previous facts to prove the similar result for the $\frac{\sin^2u}{r^2}u^2_t$ part of the energy. 

Our main tools in this section are the differential identities \eqref{abc} and \eqref{ru}, which are integrated on forward truncated cones $K(t_1,t_2)\subset \Omega$ ($0 <  t_1 \leq t_2 \leq T_0$). The strategy is to show that we can allow first $t_1\to 0$, independently of $t_2$, in the resulting equation, followed subsequently by $t_2\to 0$ in the previously obtained limit. 

We apply this approach first for \eqref{abc} with  $(a,b,c)=(t,0,0)$, deducing
\[
\int_{K(t_1,t_2)}\,e\,=\,t_2\, E(t_2) - t_1\,E(t_1)\,-\,\frac{1}{\sqrt 2}\int_{C(t_1,t_2)}\,t (e+m).\]
Using \eqref{flux} and \eqref{ne}, we obtain
\[
\lim_{t_1\to 0} \,\int_{K(0,t_1)}\,e\,+\,t_1\,E(t_1)\,+\,\int_{C(0,t_1)}\,t (e+m)\,=\,0,\]
which, coupled with the previous equality, yields
\[
\int_{K(0,t_2)}\,e\,=\,t_2\, E(t_2)\,-\,\frac{1}{\sqrt 2} \int_{C(0,t_2)}\,t (e+m).\]
Finally, using  \eqref{flux} again, we infer:
\begin{equation}
\lim_{T\to 0+}\frac{1}{T}\int_{K(0,T)}\,e\,=\lim_{T\to 0+}E(T)\,=\,l, \label{neq}
\end{equation}
which provides us with an equivalent way of showing that $l=0$.

\subsection{Behavior near the cone}
Next, our goal is to show non-concentration of energy near the light cone, i.e.,
\begin{equation}
\lim_{T\to 0+}\int_{\Sigma_\lambda(T)}\,e\,=\,0,\label{nec}
\end{equation}
for any fixed $0 < \lambda < 1$. Reasoning as in \cite{MR1168115}, it turns out that in order to claim \eqref{nec}, we need to prove the following two estimates:
\begin{align}
\left|\partial_\xi(r(e+m))\right|\,\lesssim\, ((e+m)(e-m))^{1/2}, \label{diam1}\\
\left|\partial_\eta(r(e-m))\right|\,\lesssim\, ((e+m)(e-m))^{1/2},
\label{diam2}
\end{align}
where $\eta=t+r$ and $\xi=t-r$ are the classical null coordinates.

For $(a,b,c)=(0,1,0)$, \eqref{abc} can be rewritten as
\[
\partial_t(rm)\,-\,\partial_r(re) \,=\, -\left(1 - \frac{\alpha^2\sin ^2 u}{r^2}\right)\frac{u_t^2-u_r^2}{2}  + \frac{\sin^2u}{2r^2} - \frac{\sin 2u\cdot u_r}{r}\,=\,D,\]
which, coupled with \eqref{en}, implies:
\[
\partial_\xi(r(e+m))\,=\,-\partial_\eta(r(e-m))\,=\,\frac{D}{2}.\]
Therefore, in order to prove \eqref{diam1}-\eqref{diam2}, it is enough to verify
\[
D^2\,\lesssim\,(e+m)(e-m),
\]
which follows immediately from the straightforward bound
\[
D^2\,\lesssim\,\left(1 - \frac{\alpha^2\sin ^2 u}{r^2}\right)^2 (u_t^2-u_r^2)^2  + \frac{\sin^4u}{r^4} + \frac{\sin^2 2u\cdot u^2_r}{r^2} \,\lesssim\,(e+m)(e-m).\]

As an immediate consequence of \eqref{nec}, we obtain
\begin{equation}
\lim_{T\to 0+}\int_{\Sigma(T)}\,\frac{r}{T}\,e\,=\,0,
\label{rt}
\end{equation}
which is the crucial piece of information needed to prove:

\subsection{Non-concentration of energy for the $u_t^2$ and $ \frac{\sin^2u}{r^2}u^2_r$ densities}
We integrate \eqref{abc} for $(a,b,c)=(0,r,0)$ on the truncated cone $K(t_1,t_2)$ to infer
\[
\aligned
\int_{\Sigma(t_2)}\,r m\,+\,&\int_{K(t_1,t_2)}\,u_t^2 + \frac{\alpha^2\sin^2u}{r^2}u^2_r\,=\,\int_{\Sigma(t_1)}\,r m\\
&+\,\frac{1}{\sqrt 2}\int_{C(t_1,t_2)}\,r \left[\left(1 + \frac{\alpha^2\sin ^2 u}{r^2}\right)\frac{(u_t+u_r)^2}{2}-
\frac{\sin ^2 u}{2r^2}
\right].
\endaligned
\]
Reasoning as in the start of this section, we can take $t_1\to 0$ and then divide by $t_2$ to deduce
\[
\aligned
\frac{1}{t_2}\int_{\Sigma(t_2)}\,r m\,+\,\frac{1}{t_2}&\int_{K(0,t_2)}\,u_t^2 + \frac{\alpha^2\sin^2u}{r^2}u^2_r\\
&= \frac{1}{\sqrt{2}t_2}\int_{C(0,t_2)}\,r \left[\left(1 + \frac{\alpha^2\sin ^2 u}{r^2}\right)\frac{(u_t+u_r)^2}{2}-
\frac{\sin ^2 u}{2r^2}
\right].
\endaligned
\]
Using \eqref{rt} and the trivial bounds
\[
|m|\leq e, \qquad \left|\left(1 + \frac{\alpha^2\sin ^2 u}{r^2}\right)\frac{(u_t+u_r)^2}{2}-
\frac{\sin ^2 u}{2r^2}\right| \leq e+m,\]
we conclude that
\begin{equation}
\lim_{T\to 0+}\frac{1}{T}\int_{K(0,T)}\,u_t^2 + \frac{\alpha^2\sin^2u}{r^2}u^2_r\,=\,0.\label{utsinur}
\end{equation}

\subsection{Non-concentration of energy for the $u_r^2$ density}
This is the point in the argument where we rely on \eqref{ru}, which integrated on $K(t_1,t_2)$, leads to
\begin{equation}
\aligned
\int_{\Sigma(t_2)}\,u&\, u_t\,-\,\int_{\Sigma(t_1)}\,u\, u_t\,-\,\frac{1}{\sqrt 2}\int_{C(t_1,t_2)}\, u (u_t+u_r)\\
&=\,\int_{K(t_1,t_2)}\,\left[1-\frac{\frac{\alpha^2 u\sin 2u}{2r^2}}{w}\right]\,(u_{t}^2 -u_r^2) +\left[\frac{1-\frac{\alpha^2 \sin^ 2u}{r^2}}{w}-1\right]\,\frac{uu_r}{r} -\frac{\frac{ u\sin 2u}{2r^2}}{w}.
\endaligned
\label{ur0}
\end{equation}
Based on Proposition \ref{cont}, if we choose $t_2$ sufficiently small, then $u$ is small and we can use in $K(0,t_2)$ both \eqref{dcy} and the uniform bound $|u|\lesssim |\sin u|$, which allows us morally to think of $u$ as $\sin u$, everywhere in the above equation. 

We discuss first the integrals  on the bases of the cone and on its mantle. Using the Cauchy-Schwartz inequality, for $t\leq t_2$, we obtain the following set of estimates:
\[
\aligned
\left|\int_{\Sigma(t)} \,u\, u_t\right|\,&\lesssim\, t^2 \left(\int_{\Sigma(t)}\, \frac{\sin^2u}{r^2}u^2_t \right)^{1/2}\,\lesssim\, t^2 E(t)^{1/2},\\
\left|\int_{C(0,t)} \,u\, (u_t+u_r)\right|\,&\lesssim\, t^2 \left(\int_{C(0,t)}\, \frac{\sin^2u}{r^2}(u_t+u_r)^2 \right)^{1/2}\,\lesssim\, t^2 F(0,t)^{1/2},\endaligned
\]
which imply
\begin{equation}
\lim_{T\to 0+}\frac{1}{T}\,\int_{\Sigma(T)} \,u\, u_t\,=\,\lim_{T\to 0+}\frac{1}{T}\,\int_{C(0,T)} \,u\, (u_t+u_r)\,=\,0.
\label{ur1}
\end{equation}

For the last two terms in the integral on the cone, we can argue as
\[
\aligned
\left|\int_{K(0,t)} \,\left[\frac{1-\frac{\alpha^2 \sin^ 2u}{r^2}}{w}-1\right]\,\frac{uu_r}{r} -\frac{\frac{ u\sin 2u}{2r^2}}{w}\right|\,\lesssim\,\int_{K(0,t)} \,\left|\frac{uu_r}{r}\right| + 1\\
\lesssim t^{3/2} \left(\int_{K(0,t)} \,\frac{\sin^2u}{r^2}u^2_r\right)^{1/2}\,+\,t^3,\endaligned
\]
from which we deduce
\begin{equation}
\lim_{T\to 0+}\frac{1}{T}\,\int_{K(0,T)} \,\left[\frac{1-\frac{\alpha^2 \sin^ 2u}{r^2}}{w}-1\right]\,\frac{uu_r}{r} -\frac{\frac{ u\sin 2u}{2r^2}}{w}\,=\,0.\label{ur2}
\end{equation}
Also, 
\[
\left|\int_{K(0,t)}\,\left(1-\frac{\frac{\alpha^2 u\sin 2u}{2r^2}}{w}\right)\,u_{t}^2 + \frac{\frac{\alpha^2 u\sin 2u}{2r^2}}{w}u_r^2\right|\,\lesssim\,\int_{K(0,t)}\,u_t^2 + \frac{\alpha^2\sin^2u}{r^2}u^2_r
\] 
implies, based on \eqref{utsinur},
\begin{equation}
\lim_{T\to 0+}\,\frac{1}{T}\,\int_{K(0,T)} \,\left(1-\frac{\frac{\alpha^2 u\sin 2u}{2r^2}}{w}\right)\,u_{t}^2 + \frac{\frac{\alpha^2 u\sin 2u}{2r^2}}{w}u_r^2\,=\,0.\label{ur3}
\end{equation}

Thus, taking in \eqref{ur0} first $t_1\to 0$, then dividing the limiting result by $t_2$, and finally allowing $t_2\to 0$, we conclude, based on \eqref{ur1}-\eqref{ur3}, that:
\begin{equation}
\lim_{T\to 0+}\,\frac{1}{T}\int_{K(0,T)}\,u_r^2\,=\,0.\label{ur}
\end{equation}

\subsection{Non-concentration of energy for the $ \frac{\sin^2u}{r^2}u^2_t$ density}
For this last term we use \eqref{abc} with $(a,b,c)=(0,0,1)$ and $h(u)=\sin u$ to infer, after integrating on $K(t_1,t_2)$:
\begin{equation}
\aligned
\int_{\Sigma(t_2)}\,w \sin u\,u_t\,&+\,\int_{K(t_1,t_2)}\,\left(w\cos u + \frac{\alpha^2\sin u \sin 2u}{2r^2}\right)(u_r^2-u_t^2) + \frac{\sin u \sin 2u}{2r^2}\\
&=\,\int_{\Sigma(t_1)}\,w \sin u\,u_t\,+\,\frac{1}{\sqrt 2}\int_{C(t_1,t_2)}\,w \sin u\,(u_t+u_r).
\endaligned
\label{ut0}\end{equation}

We will be working as above, with $t_2$ sufficiently small, such that we can use \eqref{dcy} and deduce for $t \leq t_2$:
\begin{equation}
\aligned
\Bigg|\int_{K(0,t)}\,\left(w\cos u + \frac{\alpha^2\sin u \sin 2u}{2r^2}\right)(u_r^2-u_t^2) &+ \frac{\sin u \sin 2u}{2r^2}\Bigg|\,\lesssim\,\int_{K(0,t)} \,e\\
&\lesssim\,t\,E(t),
\endaligned
\label{ut1}\end{equation}
\begin{equation}
\left|\int_{\Sigma(t)}\,w \sin u\,u_t\right|\,\lesssim\,\left(\int_{\Sigma(t)}\,w \sin^2 u\right)^{1/2}\,\left(\int_{\Sigma(t)}\,w \,u_t^2\right)^{1/2}\,\lesssim\,t\,E(t)^{1/2},\label{ut2}\end{equation}
\begin{equation}
\aligned
\left|\int_{C(0,t)}\,w \sin u\,(u_t+u_r)\right|\,&\lesssim\,\left(\int_{C(0,t)}\,w \sin^2 u\right)^{1/2}\,\left(\int_{\Sigma(t)}\,w \,(u_t+u_r)^2\right)^{1/2}\\
&\lesssim\,t\,F(0,t)^{1/2}.\endaligned
\label{ut3}\end{equation}
These estimates allow us to apply our general strategy (i.e., first $t_1\to 0$, then divide by $t_2$, followed by $t_2\to 0$) and obtain
\[
\aligned
\lim_{T\to 0+}\ \frac{1}{T} &\int_{\Sigma(T)}\,w \sin u\,u_t\\
&+\,\frac{1}{T}\int_{K(0,T)}\,\left(w\cos u + \frac{\alpha^2\sin u \sin 2u}{2r^2}\right)(u_r^2-u_t^2) + \frac{\sin u \sin 2u}{2r^2}\,=\,0.
\endaligned\]

Using \eqref{sinu}, \eqref{utsinur}, and \eqref{ur}, we can further strip down terms from the previous limit, which leads to
\begin{equation}
\lim_{T\to 0+}\ \frac{1}{T}\int_{\Sigma(T)}\,w \sin u\,u_t\,-\,\frac{1}{T}\int_{K(0,T)}\,\frac{2\alpha^2\sin^2 u}{r^2}\,u_t^2\,=\,0.
\label{ut4}\end{equation}
We see that \eqref{ut2} is not enough to carry the day and this is why we have to do a finer analysis of the first term in \eqref{ut4}. From \eqref{utsinur} we obtain the existence of a sequence $(t_n)_n\to 0$ satisfying \eqref{sinur}, hence we can use the finer bound \eqref{dcy3} to infer
\[
\int_{\Sigma(t_n)}\,w \sin^2 u \,\lesssim\, \int_{\Sigma(t_n)} C_n^2\, r +C_n^4 \,\lesssim\, C_n^2 \,t_n^3 +C_n^4\, t_n^2, 
\]
which implies, based on \eqref{ut2}, that
\[
\lim_{n\to \infty}\ \frac{1}{t_n} \,\int_{\Sigma(t_n)}\,w \sin u\,u_t\,=\,0.
\]
Coupling this limit with \eqref{ut4}, \eqref{neq}, \eqref{sinu}, \eqref{utsinur}, and \eqref{ur}, we conclude that
\begin{equation}
\lim_{T\to 0+}\ \frac{1}{T}\int_{K(0,T)}\,\frac{\alpha^2\sin^2u}{r^2}u_t^2\,=\,0,\label{utsinut}
\end{equation}
which finishes the proof of Theorem \ref{eth}.

\section*{Acknowledgements}
We thank Manoussos Grillakis and Sarada Rajeev for stimulating discussions during various stages of this project. Both authors was supported in part by the National Science Foundation Career grant DMS-0747656.

\bibliographystyle{amsplain}
\bibliography{anwb}

\end{document}